\documentclass[final,1p]
{elsarticle}

\usepackage{amssymb}
\usepackage{amsthm}

\usepackage{amsmath, amsthm, amssymb} 
\usepackage{graphicx}
\newtheorem{theorem}{Theorem}[section]
\newtheorem{lemma}[theorem]{Lemma}

\newtheorem{remark}[theorem]{\it Remark}

\numberwithin{equation}{section}
    \usepackage{color}  
     \definecolor{pass}{rgb}{0,0,0.8}
     \definecolor{pass1}{rgb}{0,0,0.5}
     \definecolor{blue}{rgb}{0,0,0.6}

\newcommand{\pt}{\partial}
\newcommand {\eps} {\varepsilon}

\newcommand{\R}{\mathbb{R}}

\newcommand{\LL}{{\mathcal L}}

\newcommand {\beq} {\begin{equation}}
\newcommand {\eeq} {\end{equation}}
\newcommand {\beqa} {\begin{eqnarray}}
\newcommand {\eeqa} {\end{eqnarray}}
\newcommand {\beqann} {\begin{eqnarray*}}
\newcommand {\eeqann} {\end{eqnarray*}}

\journal{}

\begin{document}

\begin{frontmatter}



\title{
Maximum principle
 for time-fractional parabolic equations
 with a reaction coefficient of arbitrary sign}


\author{Natalia Kopteva\corref{cor1}}
\cortext[cor1]{Department of Mathematics and Statistics, University of Limerick, Limerick, Ireland}
\ead{natalia.kopteva@ul.ie}

\begin{abstract}
We consider
time-fractional parabolic equations with a Caputo time derivative
of order $\alpha\in(0,1)$. 
For such equations, we give
an elementary proof of the weak maximum principle
under no assumptions on the sign of the reaction coefficient.
This proof is also extended for weak solutions, as well as for various types of boundary conditions
and variable-coefficient variable-order
multiterm time-fractional parabolic equations.
\end{abstract}







\end{frontmatter}



\section{Introduction}
\newcommand{\Nh}{\hspace{-0.35pt} }
Proofs of various versions of the maximum principle for classical elliptic and
parabolic equations 
can be found in many textbooks on PDEs 
and their applications \cite{evans,GTru,Protter,sperb}.
In
the\Nh parabolic\Nh case,\Nh
the\Nh presence\Nh of\Nh the\Nh reaction\Nh term\Nh $cu$\Nh with\Nh  $c$\Nh of\Nh arbitrary\Nh sign\Nh
 is\Nh usually addressed
by a simple solution substitution $u=e^{\mu t}\hat u$ with a sufficiently large constant $\mu$.

For time-fractional parabolic equations,
the weak maximum principle
was proved by Luchko \cite{luchko}
 in 2009
 assuming 
a nonnegative coefficient $c$ 
in an equation 
of the form \eqref{problem1} below.
%
%
However, it was not until 2017 that
this restriction on the sign of $c$ was removed
by Luchko and Yamamoto \cite{LuchYama}.
%
%
It is also worth noting that the authors,
who are unquestionably leading experts in the area,
concluded that the substitution $u=e^{\mu t}\hat u$ does not work in the time-fractional case,
%
and thus
%
they devised an alternative new proof method.
The latter method relies on eigenfunction expansions and a fixed point theorem,
so naturally applies to weak solutions, but, apart from being quite intricate,
it is limited to self-adjoint spatial operators with time-independent coefficients.

The purpose of this paper is to give an elementary proof of the weak maximum principle
for strong solutions
of time-fractional parabolic equations with  $c$ of arbitrary sign.
This new proof, surprisingly, combines the fractional-derivative representation used by Luchko in \cite{luchko}
(see \eqref{aux_lem_int}) with the substitution $u=e^{\mu t}\hat u$.
%
We shall also extend this proof for
weak solutions in $\color{blue}H^1(\Omega)$, 
as well as for
various types of boundary conditions and variable-coefficient variable-order
multiterm time-fractional parabolic equations.

We shall consider fractional-order parabolic equations, of order $\alpha\in(0,1)$, of type 
\beq
D_t^{\alpha}u+\LL u+c(x,t)\, u=f(x,t),\quad\mbox{where}\quad
c(x,t)\ge-\lambda,
\label{problem1}
\eeq
for $(x,t)\in\Omega\times(0,T]$,
subject to the initial condition $u(\cdot,0)=u_0$ in $\Omega$, and the boundary
condition $\color{blue}u=g$ on $\pt\Omega$ for $t>0$.
This problem is posed in a bounded Lipschitz domain
$\Omega\subset\R^d$ (where $d\in\{1,2,3\}$), and involves
a spatial linear second-order elliptic operator~$\LL=\LL(t)$ of the form
\beq\label{LL_def}
\LL u := -\sum_{i,j=1}^d \pt_{x_i}\!\bigl(a_{ij}(x,t)\,\pt_{x_j}\!u\bigr) + \sum_{i}^d b_i(x,t)\, \pt_{x_i}\!u,
\eeq
with a symmetric positive definite coefficient matrix $\{a_{ij}(x,t)\}_{i,j=1}^d$  $\forall (x,t)\in\Omega\times(0,T]$.
%
The Caputo fractional derivative in time, denoted 
by $D_t^\alpha$, is  defined \cite{Diet10},
for $t>0$, by
\begin{equation}\label{CaputoEquiv}
D_t^{\alpha} u := 
 \frac1{\Gamma(1-\alpha)} \int_{0}^t(t-s)^{-\alpha}\,\pt_s v(\cdot, s)\, ds,
\end{equation}
where $\Gamma(\cdot)$ is the Gamma function, and $\pt_t$ denotes the partial/ordinary derivative in~$t$.

Note that maximum principles for  time-fractional parabolic equations
of type \eqref{problem1}, with a reaction coefficient of arbitrary sign,
can be immediately applied to
extend
the new a posteriori error estimation methodology \cite{NK_AML,NK_MSt_multiterm} to the semilinear case.
It is also worth mentioning that analogous results have been used to prove existence and uniqueness of solutions of classical semilinear parabolic equations \cite{pao_book,sperb}, a direction still to be fully explored in the context of fractional parabolic equations.

\smallskip

\noindent{\it Notation.} We use the standard inner product $\langle\cdot,\cdot\rangle$ and the norm $\|\cdot\|$
in the space $L_2(\Omega)$, as well as the standard spaces 
$L_\infty(\Omega)$,\,   $H^1(\Omega)=W^{1,2}(\Omega)$,\, $H^1_0(\Omega)$,\,
$W^{1,\infty}(t',t'';\,L_\infty(\Omega))$,\, and $C([0,T]; L_\infty(\Omega))$
(see \cite[Sec.\,5.9.2]{evans} for the notation used for functions of $x$ and~$t$).
The notation $v^+:=\max\{0,\,v\}$ is used for the positive part of a generic function $v$.

\section{Maximum principle for strong solutions}

The key role in our analysis will be played by the following lemma.

\begin{lemma}\label{lem_main}
For any fixed $\lambda$ and $\mu$ such that
$0\le \lambda\le \mu^\alpha $,
let $v\in C[0,T]\cap W^{1,\infty}(\epsilon,t)$ for any $0<\epsilon<t\le T$, 
\beq\label{main_cond}
v(0)\le 0,\quad\mbox{and}\quad
D_t^\alpha v(t_0)-\lambda v(t_0)\le 0\quad\mbox{for~some}\;\; t_0\in(0,T].
\eeq
Then the function $\hat v(t):=e^{-\mu t}v(t)$ cannot attain a positive
$\max_{[0,t_0]}\!\hat v\,$
at $t=t_0$.
\end{lemma}

\begin{remark}
One important feature of
the above lemma
is that the differential inequality in \eqref{main_cond} is assumed at a single point $t_0>0$.
If it were true $\forall\,t_0>0$, then the  result
of Lemma~\ref{lem_main} (and, in fact, that $v(t)\le 0$ $\forall\,t>0$)
would follow from an explicit solution representation for $D_t^\alpha v(t)-\lambda v=w$
 (see, e.g., \cite{Diet10,sakamoto} and a discussion in  \cite[Remark~2.1]{NK_AML}).
\end{remark}

\begin{proof}
In view of \eqref{CaputoEquiv}, replacing $\pt_s v(s)$ in $D_t^\alpha v(t)$ by $\pt_s\{v(s)- v( t)\}$ and then integrating by parts
, one gets
(see also \cite{luchko,Brunner_etal_2015})
\beq\label{aux_lem_int}
\Gamma(1-\alpha)\,D_t^\alpha v(t)
=t^{-\alpha} \{v(t)-v(0)\}+\int_{0}^t\!\alpha(t-s)^{-\alpha-1}\, \{v(t)- v( s)\}\, ds.
\eeq
Next, dividing both parts by $e^{\mu t}$ and using $v(t)=e^{\mu t}\,\hat v(t)$ yields
\begin{align}
\notag
e^{-\mu t}\Bigl(\Gamma(1-\alpha)&\,D_t^\alpha v(t)+ t^{-\alpha}v(0)\Bigr)  \\ \notag
    &=t^{-\alpha}\hat  v(t)+\int_{0}^t\!\alpha(t-s)^{-\alpha-1}\, \{\hat v(t)- e^{-\mu(t- s)}\hat v(s)\}\, ds \\
    &=\bigl[t^{-\alpha}+\Lambda(t)\bigr]\,\hat  v(t)+\int_{0}^t\!\alpha(t-s)^{-\alpha-1}\,e^{-\mu(t- s)} \{\hat v(t)-\hat v( s)\}\, ds,
\label{Dt_new}
\end{align}
where
$$
\Lambda(t):=\int_{0}^t\!\alpha(t-s)^{-\alpha-1}\, \{1- e^{-\mu(t- s)}\}\, ds
=\int_{0}^t\!\alpha s^{-\alpha-1}\, \{1- e^{-\mu s}\}\, ds.
$$
It will be helpful to rewrite \eqref{Dt_new} using the  notation $\Lambda_\infty:=\Lambda(\infty)$ and
\beq\label{omega}
\omega(t):=\int_t^\infty \!\!\alpha s^{-\alpha-1} e^{-\mu s}ds.
\eeq
With this new notation, a simple calculation shows that $\Lambda_\infty-\Lambda(t)=t^{-\alpha}-\omega(t)$,
so
$t^{-\alpha}+\Lambda(t)=\omega(t)+\Lambda_\infty$, so  \eqref{Dt_new} becomes
\beq\label{Dhat_v}
e^{-\mu t}\Bigl(\Gamma(1-\alpha)\,D_t^\alpha v(t)+ t^{-\alpha}v(0)\Bigr)
=
\bigl[\omega(t)+\Lambda_\infty\bigr]\,\hat  v(t)+\int_{0}^t\!\pt_s\omega(t- s)\, \{\hat v(t)-\hat v( s)\}\, ds.
\eeq
Note that here $\omega(t)>0$, while $\pt_s\omega(t- s)>0$ (in view of $\omega'(t)<0$).
 Also, for $\Lambda_\infty$, using $\hat s:=\mu s$ and then integrating by parts, one gets
\beq\label{Lambda_def}
\Lambda_\infty=\mu^{\alpha}\alpha\int_{0}^{\infty}\!\! {\hat s}^{-\alpha-1}\, \{1- e^{-\hat s}\}\, d\hat s
=\mu^{\alpha}\int_{0}^{\infty}\!\! { s}^{-\alpha}\, e^{- s} d s=\Gamma(1-\alpha)\,\mu^\alpha.
\eeq

To complete the proof,  set $t:=t_0$ in \eqref{Dhat_v},
then subtract $e^{-\mu t_0}\Gamma(1-\alpha)\,\lambda v(t_0)=\Gamma(1-\alpha)\,\lambda \hat v(t_0)$,
and, finally, recall \eqref{main_cond}, which yields
\beq\label{main_lem_ine}
0\ge
\underbrace{\bigl[\omega(t_0)
+\Gamma(1-\alpha)\,(\mu^\alpha-\lambda)
\bigr]}_{\ge\omega(t_0)>0}\,\hat  v(t_0)+\int_{0}^{t_0}\!
\underbrace{\pt_s\omega(t_0- s)}_{>0}\, \{\hat v(t_0)-\hat v( s)\}\,ds.
\eeq
Now, assuming that $\max_{[0,t_0]} \hat v=\hat v(t_0)>0$
immediately leads to a contradiction.
\end{proof}

Now we are prepared to prove the weak maximum principle for strong solutions of \eqref{problem1}--\eqref{CaputoEquiv}
with the coefficients
$\{a_{ij}\}$ in $C^1(\Omega)$, and $\{b_i\}$ and $c$ in $C(\Omega)$.

\begin{theorem}[maximum/comparison principle]\label{th_max_pr}
Let $c(x,t)\ge -\lambda$ $\forall\,(x,t)\in\Omega\times (0,T]$ with some constant $\lambda\ge 0$.
\color{blue}
Suppose that
$u(x,0)\le 0$ $\forall\,x\in\Omega$ and $u(x,t)\le 0$
 $\forall\,(x,t)\in\pt\Omega\times [0,T]$,
and also  $u\in C(\bar\Omega\times[0,T])$, 
$u(\cdot,t)\in C^2(\Omega)$ for any $t>0$,
and   $u(x,\cdot)\in W^{1,\infty}(\epsilon,t)$ for any $x\in\Omega$ and $0<\epsilon<t\le T$.
Then $(D_t^\alpha+\LL+c(x,t))u\le 0$ in $\Omega\times(0,T]$ implies $u\le 0$ in $\bar\Omega\times[0,T]$.
\end{theorem}

\begin{proof}
It suffices to show that $u$ cannot take positive values at any point in $\Omega\times(0,T]$.
Let $\hat u:=e^{-\mu t}u$, where the constant
$ \mu^\alpha\ge  \lambda\ge 0$.
Suppose that $\max_{\bar\Omega\times[0,T]}\hat u=\hat u(x^*,t^*)>0$ for some $x^*\in\Omega$ and $t^*\in(0,T]$.
Then a standard argument using the  positive definiteness of the coefficient matrix $\{a_{ij}\}$
shows that  $\LL \hat u(x^*,t^*)\ge 0$
(see, e.g., \cite[proof of Theorem 3.1]{GTru}, \cite[Chap.\,2, Sec.\,2]{Protter}). Hence, $\LL u(x^*,t^*)\ge 0$, so
 $(D_t^\alpha+c(x^*,t^*))(x^*,t^*)\le 0$, and so $(D_t^\alpha-\lambda)u(x^*,t^*)\le 0$.
 Combining the latter with $u(x^*,0)\le 0$,
by
Lemma~\ref{lem_main}, one concludes that $\hat u(x^*,t)$ cannot attain a positive maximum on $[0,t^*]$ at $t=t^*$,
which yields a contradiction.
\end{proof}

\begin{remark}
The representation \eqref{Dhat_v} for $D_t^\alpha v$, may be considered a version of~\eqref{aux_lem_int} written in terms of $\hat v$ and using the kernel $\omega$ from~\eqref{omega}
and $\Lambda_\infty=\Gamma(1-\alpha)\,\mu^\alpha$ from \eqref{Lambda_def}.
Note that setting $\mu:=0$ yields $\omega(t)=t^{-\alpha}$ and $\Lambda_\infty=0$, so in this case \eqref{Dhat_v} is identical
with \eqref{aux_lem_int}.
On the other hand, an integration by parts in \eqref{Dhat_v} yields
\beq\label{CaputoEquiv_hat}
e^{-\mu t}\,\Gamma(1-\alpha)\,D_t^\alpha v(t)
=
\bigl[\Lambda_\infty\,\hat  v(t)-\widetilde\Lambda_\infty(t)\,\hat v(0)\bigr]+\int_{0}^t\!\omega(t- s) \,\pt_s \hat v( s)\, ds.
\eeq
Here we also used $v(0)=\hat v(0)$ and
$$
\widetilde\Lambda_\infty(t):=t^{-\alpha}e^{-\mu t}-\omega(t)
=e^{-\mu t}\!\int_0^\infty
\!\!\!\alpha (s+t)^{-\alpha-1}\, \{1- e^{-\mu s}\}\, ds\le \Lambda_\infty,
\quad \widetilde\Lambda_\infty(0)=\Lambda_\infty.
$$
Likewise, \eqref{CaputoEquiv_hat} may be considered a version of \eqref{CaputoEquiv} written in terms of $\hat v$.
\end{remark}

\section{Maximum principle for weak solutions}

In this section, we shall extend the above maximum principle to the case of weak solutions in $W^{2,1}(\Omega)=H^1(\Omega)$.
For this purpose, we borrow some definitions and general ideas from \cite[Chap.\,8]{GTru}
 (where they are used to establish the weak maximum principle for functions
in $H^1(\Omega)$ in the context of elliptic equations).

Throughout this section, it will be assumed for the elliptic operator $\LL$ from \eqref{LL_def} that its coefficients are in $L_\infty(\Omega\times(0,T))$,
and that
\beq\label{C_LL}
C_{\LL}:=\sup_{\Omega\times(0,T)}\frac{\sum_{i=1}^d b_i^2}{4\lambda_a}
<\infty
\eeq
where $\lambda_a=\lambda_a(x,t)>0$ is the minimal eigenvalue of the
symmetric positive definite
coefficient matrix $\{a_{ij}(x,t)\}$.

Now, $u(\cdot, t)\in \color{blue}H^1(\Omega)$ is said to satisfy \eqref{problem1} in a weak sense if
$\langle (D_t^{\alpha}+\LL + c)u,\,\chi\rangle=\color{blue}\langle f,\,\chi\rangle$
 $\forall \chi\in H_0^1(\Omega)$ and $\forall t\in(0,T]$.
Similarly, $u$ is said to satisfy
$(D_t^{\alpha}+\LL + c)u\le 0$ $(\ge 0)$ in a weak sense if $\langle (D_t^{\alpha}+\LL + c)u,\,\chi\rangle\le 0$ $(\ge 0)$
for all nonnegative $\chi\in H_0^1(\Omega)$.
{\color{blue}Additionally, $w \in H^1(\Omega)$ is said to satisfy $w\le 0$  on $\pt\Omega$
if $w^+\in H^1_0(\Omega)$.
}

\begin{theorem}[maximum principle for weak solutions]\label{th_weak}
Let $\inf_{\Omega\times(0,T)}c\ge -\lambda$ 
for some constant $\lambda\ge 0$.
Suppose that
{\color{blue}$u(\cdot,0)\le 0$ in $\Omega$ and $u(\cdot,t)\le 0$ on $\pt\Omega$ for any
 $t\in [0,T]$,}
and also
 $u$ is in
$C([0,T]; L_\infty(\Omega)) \cap W^{1,\infty}(\epsilon,t;\,L_\infty(\Omega))$ for any $0<\epsilon<t\le T$,
while
$u(\cdot,t)\in H^1(\Omega)$ for any $t>0$.
Then $(D_t^\alpha +\LL-c)u\le 0$, understood in a weak sense, implies
\color{blue}$\sup_{\Omega}u(\cdot,t)\le 0$
$\forall\,t\in [0,T]$.
\end{theorem}

To prove the above theorem we shall employ the following auxiliary lemma.

\begin{lemma}\label{lem_pos_def}
  Suppose that $w\in L_\infty(\Omega)\cap \color{blue}H^1(\Omega)$
  and
  {\color{blue}$w\le 0$ on $\pt\Omega$}, and, additionally,
  $0<M'<\sup_{\Omega}w\le M$ for some constants $M'$ and $M$.
  Then for the function $\chi_0:=(w-M')^+$,
  one has
  \beq\label{chi_C_ll}
  \langle\LL w,\chi_0\rangle\ge -C_{\LL}(M-M')\,\langle1,\chi_0\rangle,
  \eeq
  where
  $\LL=\LL(t)$ for any fixed $t\in(0,T]$ and
  $C_{\LL}$ is from \eqref{C_LL}.
\end{lemma}

\begin{proof}
First, note that
$\chi_0=\max\{0,\,w-M'\}$ is in $ H_0^1(\Omega)$
{\color{blue}(in view of $w\le 0$ on $\pt\Omega$)}
and has a nonempty support $\Omega_0\subset\Omega$.
Also, $\chi_0$ takes values in $[0, M-M']$, while
$\pt_{x_i}\chi_0=\pt_{x_i}\!w$ $\forall i$ in $\Omega_0$
(see also a more detailed discussion in \cite[Sec.\,8.1]{GTru}).

Now, applying an integration by parts
to the first sum in $\langle\LL w,\chi_0\rangle$,
and then using
the  positive definiteness of the coefficient matrix $\{a_{ij}\}$ and \eqref{C_LL},   yields
\begin{align}\notag
 \langle\LL w,\chi_0\rangle&=
\sum_{i,j=1}^d \int_{\Omega_0}\!\!a_{ij}\,(\pt_{x_j}\!w)\,(\pt_{x_i}\chi_0)  + \sum_{i}^d \int_{\Omega_0}\!\!  (b_i\,\pt_{x_i}\!w)\,\chi_0 \\
&\ge\int_{\Omega_0}\!\!\lambda_a|\nabla w|^2
-\int_{\Omega_0}\!\!2(\lambda_a C_{\LL})^{1/2}|\nabla w|\,|\chi_0|
\notag\\[0.3pt]&
\ge -C_{\LL}\|\chi_0\|^2.
\label{by_parts}
\end{align}
It remains to note that $\|\chi_0\|^2=\langle\chi_0,\chi_0\rangle \le (M-M')\,
\langle1,\chi_0\rangle $.
\end{proof}

\noindent{\it Proof of Theorem~\ref{th_weak}.}~
%
Let $\hat u:=e^{-\mu t}u$, where the constant
$\mu^\alpha\ge \lambda\ge 0  $.
%
As
$u\in C([0,T]; L_\infty(\Omega))$, so
$S(t):=\sup_{\Omega}\hat u(\cdot, t)$ is in $  C[0,T]$, with $S(0)\le 0$,
and it  suffices to prove that $S(t)\le 0$ $\forall t$.

Suppose that $\max_{t\in[0,T]}S(t)=S(t_0)=\sup_{\Omega}\hat u(\cdot, t_0)=M>0$ for some $t_0\in(0,T]$.
Next, for any arbitrarily small $\varepsilon\in(0,\frac12 M)$, set $M':=M-\varepsilon>0$
and
consider
$\chi_0:= (\hat u(\cdot, t_0)-M')^+$, which has a non-empty support $\Omega_0\subset\Omega$.
By Lemma~\ref{lem_pos_def},
one gets \eqref{chi_C_ll} with $w:=\hat u(\cdot, t_0)$ and $\LL=\LL(t_0)$.
So for $\chi:=\chi_0/\langle1,\chi_0\rangle$, one gets
$\langle\LL w,\chi\rangle \ge -C_{\LL}(M-M')$, or, equivalently,
$\langle\LL (t_0)\hat u(\cdot, t_0),\chi\rangle \ge -\varepsilon C_{\LL}$.
Also,
$(D_t^{\alpha}+\LL + c)u\le 0$  implies
$$
\langle (D_t^{\alpha}+\LL(t_0) + c(\cdot, t_0))u(\cdot, t_0),\,\chi\rangle\le 0
$$
for the nonnegative $\chi\in H_0^1(\Omega)$.
As $\hat u(\cdot, t_0)$ is nonnegative  in $\Omega_0$, so is $ u(\cdot, t_0)$,
so
$\langle c(\cdot, t_0)\,u(\cdot, t_0),\,\chi\rangle\ge -\lambda \langle u(\cdot, t_0),\,\chi\rangle$.
Combining these observations,
one concludes that 
$\langle D_t^{\alpha} u(\cdot, t_0),\chi\rangle -e^{\mu t_0}\,\varepsilon C_{\LL}-\lambda\langle u(\cdot, t_0),\chi\rangle\le 0.$

Hence, setting $v(t)=\langle  {\color{blue}u}(\cdot,t),\chi\rangle$ and $\hat v(t):=\langle \hat u(\cdot,t),\chi\rangle=e^{-\mu t}\,v(t)$
yields
\beq\label{main_cond_no}
D_t^\alpha v(t_0) -\lambda v(t_0)\le e^{\mu t_0}\,\varepsilon C_{\LL}, \qquad\quad v(0)\le 0.
\eeq
Note also that both $v$ and $\hat v$ are in $ C[0,T]\cap W^{1,\infty}(\frac12 t_0, t_0)$.
Additionally, in view of $\hat u(\cdot,t_0)\ge M'=M-\varepsilon$ in $\Omega_0$ and $\langle 1,\chi\rangle =1$,
one gets
$$
\hat v(t_0)\ge M-\varepsilon, \quad \{\hat v(t_0)-\hat v( s)\}\ge -\varepsilon\;\;\forall\,s, \quad
\{t_0-s\}^{-1}|\hat v(t_0)-\hat v( s)|\le C_{0}\;\;\forall\,s\in({\textstyle\frac12} t_0,t_0),
$$
for some constant $C_0=C_0(t_0)$.
{\color{blue}(Importantly, $C_0$ depends only on the norm of $\hat u$ in ${W^{1,\infty}({\textstyle\frac12} t_0,t_0;\,L_\infty(\Omega))}$, so is independent of $\chi$, or, equivalently, of $\eps$.)}
In view of the similarity of  \eqref{main_cond_no} with \eqref{main_cond},
an inspection of the proof of
Lemma~\ref{lem_main} shows that now we have a version of
\eqref{main_lem_ine}
with $0$ in the left-hand side replaced by
$\Gamma(1-\alpha)\,\varepsilon C_{\LL}$, i.e.
$$
\Gamma(1-\alpha)\,\varepsilon C_{\LL}\ge
\omega(t_0)\,\hat  v(t_0)+\int_{0}^{t_0}\!
\pt_s\omega(t_0- s)\, \{\hat v(t_0)-\hat v( s)\}\,ds.
$$
Here, recalling the definition \eqref{omega} of $\omega(t)$, for the integral kernel one has
$$
\int_0^{t_0-\varepsilon}\!\!\!\!\!
\pt_s\omega(t_0- s)\,ds
\le \omega(\varepsilon)\le \varepsilon^{-\alpha},
\;
\int_{t_0-\varepsilon}^{t_0}\!\!
\{t_0-s\}\,\pt_s\omega(t_0- s)\,ds=\int_0^\varepsilon\!\! \alpha s^{-\alpha}e^{-\mu s}ds\le
\frac{\alpha\varepsilon^{1-\alpha}}{1-\alpha}.
$$
So combining the above observations yields
$$
\Gamma(1-\alpha)\, \varepsilon C_{\LL}\ge
\omega(t_0)\,(M-\varepsilon)
-\varepsilon\, \varepsilon^{-\alpha}
-C_0\frac{\alpha}{1-\alpha}\varepsilon^{1-\alpha}.
$$
As $\varepsilon$ here is arbitrarily small, one concludes that $0\ge \omega(t_0)\,M$, or $M\le 0$.
%
\hfill$\square$
\bigskip

\section{Further extensions}

\subsection{Periodic and mixed boundary conditions}

Theorem~\ref{th_weak} applies to the case of Neumann/Robin boundary condition of the form 
${\mathcal B} u+g\, u:=\sum_{i,j=1}^d \nu_i\,a_{ij}\,\pt_{x_j}\!u+g\, u\le 0$ on $\pt\Omega$
(instead of $u\le 0$)
with any weight $g\ge 0$,
where $\{\nu_i\}_{i=1}^d$ are the Cartesian coordinates of the outward normal vector to $\pt\Omega$.
Note that ${\mathcal B} u+g\, u\le 0$ on $\pt\Omega$ combined with
$(D_t^{\alpha}+\LL + c)u\le 0$ in $\Omega$
for a weak solution $u$
 is  understood in the sense that
$\color{blue}\langle (D_t^{\alpha}+\LL + c)u,\,\chi\rangle+\int_{\pt\Omega}({\mathcal B} u+g\, u)\chi\le 0$
for all nonnegative $\chi\in H^1(\Omega)$.

{\color{blue}For this case, we require a version of Lemma~\ref{lem_pos_def} with
the assumption $w\le 0$ on $\pt\Omega$ dropped, and
\eqref{chi_C_ll} modified to
  $\langle\LL w,\chi_0\rangle+\int_{\pt\Omega}({\mathcal B} w+g\, w)\chi_0\ge -C_{\LL}(M-M')\,\langle1,\chi_0\rangle$.
  Note that $\chi_0$ is no longer in $H^1_0(\Omega)$, which leads to a minor change in the proof  in that the integration by parts in the bound \eqref{by_parts} for
$\langle\LL w,\chi_0\rangle$
yields an additional term
$ \int_{\pt\Omega}(-{\mathcal B}w)\chi_0$. So once $\int_{\pt\Omega}({\mathcal B} w+g\, w)\chi_0$ is added, the bound includes one additional
term $ \int_{\pt\Omega}(gw) \chi_0\ge 0$, which is nonnegative, so can be dropped.
}


Similarly, Theorem~\ref{th_weak} is easily extended for mixed and periodic boundary conditions.

\subsection{Multiterm time-fractional parabolic problem}\label{ssec_multiterm}
Our results seamlessly extend to  variable-coefficient
multiterm time-fractional parabolic equations  (see \cite{NK_MSt_multiterm} and references therein) of the form
\beq\label{prob_multi}
\sum_{i=1}^{\ell}q_i(t)\, D _t ^{\alpha_i} u(x,t)+ \LL u+c(x,t)\, u=f(x,t)
,\quad\mbox{where}\quad
c(x,t)\ge-\lambda,
\eeq
with
$0<\alpha_\ell<...<\alpha_2< \alpha_1 \le 1$ for some positive integer $\ell$,
under the conditions $q_i(t)\ge 0$, $i=1,\ldots, \ell$, and $\sum_{i=1}^{\ell}q_i(t)\ge C_q>0$  $\forall\, t\in[0,T]$.
Setting $\mu>\max_{i=1\ldots, \ell}\{C_q^{-1}\lambda\}^{1/\alpha_{i}}$
guarantees that $\lambda< \sum_{i=1}^{\ell} q_i \mu^{\alpha_{i}}$
{\color{blue}(with the strict inequalities used here to accommodate a possible $\alpha_1=1$)}.
Then Lemma~\ref{lem_main} holds true with $D_t^\alpha$ replaced by
$D_t^{\bar\alpha}:=\sum_{i=1}^{\ell}q_i(t_0)\, D _t ^{\alpha_i}$ (and, if $\alpha_1=1$, under an additional assumption that $v\in C^1(0,T]$).
Hence, if $\alpha_1<0$, one immediately gets versions of Theorems~\ref{th_max_pr} and~\ref{th_weak},
while if $\alpha_1=1$, one only gets Theorem~\ref{th_max_pr} under the additional assumption that
$u(x,\cdot)\in C^{1}(0,T]$ for any $x\in\Omega$.

\subsection{Variable-order time-fractional parabolic equations}
Variable-order time-fractional equations
have received a
lot of attention in recent years (see \cite{Hong} and its references).
%
One example from \cite{Hong} involves $D_t^{\bar\alpha(t)}=\pt_t u+q(t) D_t^{\alpha(t)}$,
using  \eqref{CaputoEquiv} with a variable
$\alpha=\alpha(t)\in(0,1)$ and $q(t)\ge0$.
%
Note that Lemma~\ref{lem_main} immediately applies to this case under the condition that
$\color{blue}0\le \lambda< \mu+q(t_0)\mu^{\alpha(t_0)}$
(and with the obvious change
$D_t^{\bar\alpha(t_0)} v(t_0)-\lambda v(t_0)\le 0$ in \eqref{main_cond}).
%
Furthermore, all conclusions of \S\ref{ssec_multiterm} apply to a variable-order version of \eqref{prob_multi}
with $\alpha_i=\alpha_i(t)$, $i=1,\ldots,\ell$.

\section*{Acknowledgements}
{\color{blue}The author thanks 
Martin Stynes, who provided many useful comments during the preparation of this paper,
 and the anonymous referees for their helpful and constructive suggestions.
 This research was partially supported by  Science Foundation Ireland under Grant number 18/CRT/6049.}

\end{document}